\documentclass[a4paper,12pt]{amsart}

\usepackage{t1enc}

\usepackage{amsmath,amsfonts,amssymb}
\usepackage{bm} 

\usepackage[usenames,dvipsnames]{xcolor}
\usepackage{tikz-cd}
\usetikzlibrary{matrix,arrows,calc}
\usepackage{calc}

\usepackage{booktabs} 
\usepackage{nicematrix} 

\usepackage[colorlinks=true,linkcolor=blue,citecolor=red,urlcolor=Violet,bookmarksdepth=subsection,final]{hyperref}

\tikzset{ampersand replacement=\&}

\def\bg{\bm{g}}
\def\bep{\bm{\epsilon}}

\def\id{\mathrm{id}}
\newcommand{\newd}{\tilde{d}}
\newcommand{\im}{\operatorname{im}}

\newcommand{\Hop}{H}
\newcommand{\Cop}{C}
\newcommand{\Dop}{D}
\newcommand{\Kop}{K}
\newcommand{\tH}{\tilde{\Hop}}

\def\cE{\mathcal{E}}
\def\cV{\mathcal{V}}
\def\cW{\mathcal{W}}
\def\cH{\mathcal{H}}

\def\si{\sigma}
\def\ta{\tau}
\def\ph{\varphi}
\def\Ga{\Gamma}
\def\La{\Lambda}
\def\Ph{\Phi}
\def\Rho{\operatorname{P}}
\def\Up{\Upsilon}
\def\na{\nabla}

\newcommand{\wh}[1]{\widehat{#1}}
\newcommand{\ol}[1]{\overline{#1}}
\newcommand{\wt}[1]{\widetilde{#1}}

\newtheorem{theorem}{Theorem}[section]
\newtheorem{definition}[theorem]{Definition}
\newtheorem{lemma}[theorem]{Lemma}
\newtheorem{proposition}[theorem]{Proposition}

\theoremstyle{remark}
\newtheorem{remark}[theorem]{\rm\bf Remark}

\begin{document}

\title[Compatibility complexes for the C2E operator]{Compatibility complexes for the conformal-to-Einstein operator}

\author{Igor Khavkine}
\address{Institute of Mathematics of the Czech Academy of Sciences}
\curraddr{\v{Z}itn\'a 25, 110 00 Prague 1, Czech Republic}
\email{khavkine@math.cas.cz}

\author{Josef \v{S}ilhan}
\address{Institute of Mathematics and Statistics, Masaryk University}
\curraddr{Building 08, Kotl\'a\v{r}sk\'a 2, 611 37 Brno, Czech Republic}
\email{silhan@math.muni.cz}

\thanks{\emph{Funding}:
IK's was partially supported by GA\v{C}R project GA22-00091S and the \emph{Czech Academy of Sciences} Research Plan RVO: 67985840.
J\v{S} was supported by GA\v{C}R project GA22-00091S}

\begin{abstract}
The conformal-to-Einstein operator is a conformally invariant linear
overdetermined differential operator whose non-vanishing solutions
correspond to Einstein metrics within a conformal class. We construct
compatibility complexes for this operator under natural genericity
assumptions on the Weyl curvature in dimension $n\ge 4$, which implies
at most one independent solution. An analogous
result for the projective-to-Ricci-flat operator is obtained as well.
The construction is based on a method, previously proposed by one of the
authors, that leverages existing symmetries and geometric properties of
the starting operator. In this case the compatibility complexes
consist of, respectively, conformally and projectively invariant
operators. We also make some comments on how Bernstein-Gelfand-Gelfand
sequences can be interpreted as compatibility complexes
in the locally flat case, which may be of
general interest.
\end{abstract}



\maketitle

\section{Introduction} \label{s.intro}

The study of conformally related Einstein metrics is a classical problem~\cite{Brinkmann}
as well as an active area of recent research, see, e.g.,\ the monograph~\cite{Besse} 
or discussion of curvature obstructions~\cite{GoNu}.
Conformal geometry $(M,[g])$ on a smooth manifold $M$ is a class of conformally related metrics (of any signature)
$[g] = \{ e^{2\Upsilon} g\}$ for $\Upsilon \in C^\infty(M)$ positive everywhere and $n = \dim M \geq 3$.
Einstein metrics in $[g]$ are locally controlled by the so-called \emph{conformal-to-Einstein} operator
\begin{equation} \label{aE}
E_0 \colon \cE[1] \to \cE_{(ab)_0}[1], \quad \si \mapsto \bigl( \na_{(a} \na_{b)_0} + \Rho_{(ab)_0} \bigl) \si.
\end{equation}
This operator is linear, conformally invariant, overdetermined, and solutions $\si \in \cE[1]$ without zeros are in 1-1 correspondence with Einstein metrics $\hat{g}$ in $[g]$.
Specifically, if we identify $\si$ with a function using the Riemannian volume form of $g$ then 
$\hat{g} = \sigma^{-2} g$.
We use the notation $\cE[w]$, $w \in \mathbb{R}$ for the bundle of ordinary $(-\frac{w}{n})$-densities, 
i.e.,\ $\La^n T^*M = \cE[-n]$, 
$\na$ is the Levi-Civita connection of $g$, $(-)_0$ denotes the trace-free part and $\Rho_{ab}$ 
denotes the Schouten tensor of $g$. We also use the abstract index notation, i.e.,\ $\cE_a = T^*M$, $\cE_{(ab)} = S^2T^*M$, $\cE_{[ab]} = \La^2 T^*M$ etc.
We put $\cE_a[w] := \cE_a \otimes \cE[w]$ and similarly for more complicated tensor bundles.
More generally, all differential operators that we will consider below
will be linear, with smooth coefficients, and mapping between some
vector bundles over some fixed manifold $M$. We will denote the
composition of two differential operators $K$ and $L$ by $K\circ L$, or
simply by $K L$, if no confusion is possible.

After the obvious question of finding the solutions of a differential equation
such as $E_0(\sigma) = 0$, another important question is to determine the range
of the operator $E_0$. For instance, analyzing the image of the exterior derivative operator $d$ on forms leads in a sense to the de~Rham complex and de~Rham cohomology. 
In another context, when analyzing the symmetries of a
geometric PDE, such as the Laplace or Dirac equations, intermediate
calculations run into the question of whether a certain tensor is in the image
of a given differential operator. Specific examples lead, e.g.,\ the exterior derivative on functions~\cite{MRS} or the conformal Killing operator on 
vector fields~\cite{ABB}. 
In part, we are motivated to look at the operator $E_0$ as a warm-up to studying similar more complicated operators, as approached not from the direction of maximal symmetry (or flatness), but from the direction of more generic geometries.

According to the theory of
overdetermined PDEs~\cite{spencer, bcggg}, the image of a linear differential operator
like $E_0$ acting on smooth functions can be locally characterized as the
kernel of another differential operator, say $E_1$, its so-called
\emph{compatibility operator} (cf.~Definition~\ref{def:compat} for the precise
notion that we will use below). When dealing with compatibility operators, it
is actually convenient to determine at the same time a whole sequence of them,
$E_k$ being the compatibility complex of $E_{k-1}$, the so-called
\emph{compatibility complex}. In this work, we will explicitly construct a
compatibility complex for the conformal-to-Einstein operator $E_0$
(Section~\ref{sec:c2e-compat}) and a closely related operator from projective
geometry (Section~\ref{sec:p2e-compat}) under certain curvature assumptions.

An explicit construction of compatibility complexes is mostly known only for a geometrical setting 
with symmetries. A famous example is the Calabi complex for the Killing operator on Riemannian metrics 
of constant sectional curvature~\cite{calabi, eastwood-calabi, kh-calabi}.
From a modern point of view, this is an example of the Bernstein-Gelfand-Gelfand (BGG)
complex in the flat projective geometry. More generally, we can consider BGG sequences in parabolic geometries
\cite{CSS,CD}. In the locally flat case, any such sequence provides a compatibility complex of the first operator (cf.\ Section~\ref{s.cc-flat}).
These (overdetermined) operators are known as \emph{first BGG operators} and~\eqref{aE} is a prominent example.
A BGG sequence will not provide a compatibility complex beyond the locally flat case, however, it will even fail to be a complex.
In fact, almost nothing is known about compatibility complexes for first BGG operators in curved parabolic 
geometries. Compatibility complexes for Killing operators on certain (pseudo)Riemannian manifolds of physical and mathematical interest in General Relativity are known~\cite{kh-compat, aabkhw}; they were in fact constructed by the same methods that we will use in this work. 
Furthermore, a partial result was obtained  in~\cite{celm-riemann, celm-lorentz} 
where the first two operators  of a compatibility complex (i.e.,\ the Killing operator 
and its compatibility operator) were obtained on locally symmetric spaces. 

Our main result is a compatibility complex for $E_0$ in~\eqref{aE} in suitably generic cases as follows. Assuming $n \geq 4$,
the conformal geometry $(M,[g])$   is locally flat if and only if the Weyl tensor 
$W_{abcd}$ vanishes. As the `opposite' case, we shall call $(M,[g])$ 
\emph{generic} if the bundle map $\cE^d \to \cE_{abc}$
induced by $W_{abcd}$ is injective. Then the solution space of~\eqref{aE} is at most 1-dimensional, 
cf.\ Section~\ref{s.generic} for details.

\begin{theorem} \label{main}
Assume $(M,[g])$ is generic and the solution space of~\eqref{aE} is 1-dimensional. Then 
there is a compatibility complex for $E_0$ of the form
$$
\begin{tikzcd}[column sep=1.5cm]
	\cE[1] \ar{r}{E_0}
		\&
	\cE_{(ab)_0}[1]
		\ar{r}{\left( \begin{smallmatrix} \id - D_1 C_1 \\ \newd \, C_1 \end{smallmatrix} \right)}
		\&
	\begin{smallmatrix} \cE_{(ab)_0}[1] \\ \oplus \\ \cE_{[ab]}[1] \end{smallmatrix}
		\ar{r}{\left(\begin{smallmatrix} C_1 & -H'_1 \\ 0 & \newd \end{smallmatrix}\right)}
		\&	
	\begin{smallmatrix} \cE_a[1] \\ \oplus \\ \cE_{[abc]}[1] \end{smallmatrix}
	\\
		\&
		\ar{r}{\big(\begin{smallmatrix} 0 & \newd \end{smallmatrix}\big)}
		\&
	\cE_{[abcd]}[1] .
		\ar{r}{\newd}
		\&
	\cE_{[abcde]}[1] \rlap{~$\cdots$.}
\end{tikzcd}
$$
Here the operators 
\begin{align*}
	C_1\colon& \cE_{(ab)_0}[1] \to \cE_a[1], &
	D_1\colon& \cE_a \to \cE_{(ab)_0}[1], \\
	\tilde{d}\colon& \cE_{[a_1 \ldots a_k]}[1] \to \cE_{[a_0 \ldots a_k]}[1], &
	H'_1\colon& \cE_{[ab]}[1] \to \cE_a[1]
\end{align*}
are defined in Section~\ref{s.invariant}. All these operators are
conformally invariant. The rest of the complex is just the twisted
exterior $\newd$ acting on forms twisted with $\cE[1]$.
\end{theorem}

In addition, we also obtain a compatibility complex when $E_0$ has 0-dimensional solution space
in Section~\ref{sec:nosol}.
This complex is much shorter -- there are only three operators (including $E_0$).
The latter case can be also adapted to dimension 3. This in particular means
that the problem of a compatibility complex for $E_0$ is fully solved in dimension 3 for
any signature and dimension 4 in the Riemannian signature, cf.\ discussion
in Remark~\ref{lowdim}.

In Section~\ref{sec:conformal} we quickly summarize our conventions and relevant facts from conformal geometry. Then, Section~\ref{s.c2e} presents all properties of the conformal-to-Einstein operator $E_0$ that will be relevant for our main results. Section~\ref{sec:compat} briefly recalls the basic ideas about compatibility complexes and summarizes the construction method from~\cite{kh-compat}, which works for operators that are uniformly of finite type, of which our $E_0$ and its projective analog are examples. Finally, in Section~\ref{sec:c2e-compat} we give the proof of our main result (Theorem~\ref{main}) for \emph{generic} conformal structures, while in Section~\ref{sec:p2e-compat}, we obtain an analogous result for the projective analog of $E_0$ (Theorem~\ref{thm:onesol-proj-cc}).

\section{Conformal geometry} \label{sec:conformal}

Let $(M,[g])$ be a conformal manifold of the dimension $n \geq 3$ and any  
signature $(p,q)$ where $n = p+q$. We shall use the notation mentioned in the beginning of Section~\ref{s.intro}.
The conformal structure 
gives rise to the \emph{conformal metric} $\bg_{ab} \in \cE_{(ab)}[2]$ with the inverse $\bg^{ab} \in \cE^{(ab)}[-2]$.
We have $\na_a \bg_{bc}=0$ for all metrics (and corresponding Levi-Civita connections $\na$) in $[g]$.
Further we have the conformal volume form $\bep_{a_1 \ldots a_n} \in \cE_{[a_1 \ldots a_n]}[n]$ which is also parallel.

Denoting by $\na$ and $\hat{\na}$ the Levi-Civita connection corresponding to $g_{ab}$ and 
$\hat{g}_{ab} = e^{2\Up}g_{ab}$,
respectively, and putting $\Up_a = \na_a \Up$, we have
\begin{equation} \label{na-trans}
\begin{split}
& \hat{\na}_a \si = \na_a \si + w \Up_a \si, \\
& \hat{\na}_a \ta_b = \na_a \ta_b - \Up_a \ta_b - \Up_b \ta_a + \Up^r \ta_r \bg_{ab}
\end{split}
\end{equation}
where $\si \in \cE[w]$ and $\ta_b \in \cE_b$. An analogous relation between $\hat{\na}_a$ and $\na_a$
on general tensor bundles can be computed using the Leibniz rule.
We denote by $R_{ab}{}^c{}_d$ the Riemann curvature tensor of $\na$, by $W_{ab}{}^c{}_d$ its Weyl tensor, by 
$Ric_{ab} = (n-2)P_{ab} + g_{ab} J$ the Ricci tensor, which also gives its relation to the Schouten tensor $P_{ab}$ and its trace $J = P^a{}_a$ its trace, and by $Y_{abc} = 2 \nabla_{[b} P_{c]a}$ the Cotton tensor.
Note the Weyl tensor is independent of the choice of the metric from $[g]$.
Note $(M,[g])$ is locally flat for $n \geq 4$ if and only if $W_{abcd}=0$ and, for $n=3$, if and only if $Y_{abc}=0$.

\section{Conformal-to-Einstein operator} \label{s.c2e}

We study the system of PDEs given by the 2nd order linear differential operator
\begin{equation} \label{E}
E=E_0: \cE[1] \to \cE_{(ab)_0}[1], \quad \si \mapsto (\na_{(a}\na_{b)_0} + P_{(ab)_0})\sigma
\end{equation}
where the subscript $0$ indicates the trace-free part with respect to
$\bg_{ab}$. This operator is conformally invariant and its non-vanishing
solutions provide rescaling to  Einstein metrics in $[g]$. That is, $[g]$ contains an Einstein metric
iff $E$ has at least one solution (without zeros).

Our aim is to construct a compatibility complex for~\eqref{E}. 
This seems to be difficult in general so we shall consider
only two special cases (which are rather opposite): the locally flat case and  the `generic case'.
In the former case, the well known BGG resolution 
provides the compatibility complex (cf.~Section~\ref{s.cc-flat}). Our main interest is the generic case; 
this is made precise in Section~\ref{s.generic}.

\subsection{Conformally flat case.} \label{s.BGG}
The operator $E_0$ is the first in the sequence of conformally  invariant operators $E_0, \ldots, E_{n-1}$ where 
$E_{n-1}$
is the formal adjoint of $E_0$ and remaining operators are of the first order,
\begin{equation} \label{BGGseq}
\begin{aligned}
E_k: \cE_{[a_1 \ldots a_k]} \boxtimes \cE_b[1] &\to \cE_{[a_0 \ldots a_k]} \boxtimes \cE_b[1],
	\\
\tau_{a_1 \ldots a_kb} &\mapsto (k+1)\Pr{}_{\boxtimes} \na_{[a_0} \tau_{a_1 \ldots a_k]b}
\end{aligned}
\end{equation}
for $1 \leq k \leq n-2$. Here $\boxtimes$ denotes the Cartan product, i.e.,\ the bundle 
$\cE_{[a_1 \ldots a_k]} \boxtimes \cE_b$ is given by the Young tableaux
with one column of the length $k$ and second column of the length $1$
where we also project to the trace-free part with respect to $\bg_{ab}$;
$\Pr{}_{\boxtimes}$ is exactly the projection to that bundle. 
In particular, $\cE_a \boxtimes \cE_b = \cE_{(ab)_0}$.

This sequence can be obtained form the exterior derivative twisted with the so called tractor connection
-- which is flat in the locally flat case -- via technique known as `BGG machinery'. 
Assuming the locally flat setting, the tractor connection is flat hence and the sequence is known as the  
\emph{BGG complex}, which is then known to be a compatibility complex in
the sense of Definition~\ref{def:compat}. 
We refer to Section~\ref{s.cc-flat} for details.

\subsection{Generic conformal structures.} \label{s.generic}
Assuming $n \geq 4$, we shall call the conformal class $(M,[g])$ \emph{generic}%
	\footnote{Called \emph{weakly generic} in~\cite[Sec.2.5]{GoNu}.} %
if the Weyl tensor is injective as the map
\begin{equation} \label{generic}
W_{abcd}\colon \cE^d \to \cE_{abc}[2], \quad v^d \mapsto W_{abcd}v^d.
\end{equation}
Then there exists an inversion $\ol{W}{}^{abce} \in \cE^{[ab]ce}[-2]$,
$$
\ol{W}{}^{abce}: \cE_{abc}[2] \to \cE^e
\quad \text{such that} \quad \ol{W}{}^{abce}W_{abcd} = \delta_d^e.
$$
Note such an inversion is not given uniquely. In that case, a specific
choice would have to be used as an extra geometric input to the rest of
our discussion. However, sometimes, there is a preferred choice. For
instance~\cite[Sec.2.5]{GoNu}, if the tensor
\begin{equation} \label{eq:V-from-W}
	V_a^b := W_{cdea} W^{cdeb}
\end{equation}
is invertible as an endomorphism (equivalently $\det V \ne 0$), with
inverse $\ol{V}{}_a^b V_b^c = \delta_a^c$, we can use the preferred choice
\begin{equation} \label{eq:Winv-from-V}
	\ol{W}{}^{abcd} := W^{abce} \ol{V}{}_e^d .
\end{equation}

The genericity assumption leads~\cite{GoNu} to explicit curvature obstructions for existence of solutions
of $E_0$ which we now briefly summarize. Using the operator $E_1$ from Section~\ref{s.BGG}, explicitly 
given by $E_1(\tau)_{abc} = 2\Pr_{\boxtimes}(\na_{[a}\tau_{b]c})$ for $\tau_{ab} \in \cE_{(ab)_0}[1]$, we compute
\begin{equation} \label{comp}
F_{abc} := (E_1 \circ E_0 (\si))_{abc} = W_{abcd}(\na^d \si) +  Y_{cab} \si
	\in \bigl( \cE_{[ab]} \boxtimes \cE_{c}\bigr)[1] .
\end{equation}
Applying the inverse $\ol{W}{}^{abce}$ to~\eqref{comp}, we obtain
\begin{equation} \label{compinv}
\na_a \si + Z_a \si = \ol{W}{}^{rst}{}_a F_{rst} \quad \text{for} \quad 
Z_a := \ol{W}{}^{rst}{}_a Y_{trs} \in \cE_a.
\end{equation}
Note the 1-form $Z$ -- determined by the choice of the inversion $\ol{W}{}^{abcd}$ -- is not conformally invariant.
If we change the metric to $\hat{g}_{ab} = e^{2\Up}g_{ab}$ , the Cotton tensor transforms as
$\wh{Y}_{bcd} = Y_{bcd} + \Up^a W_{abcd}$ where $\Up_a = \na_a \Up$. Therefore 
\begin{equation} \label{Ztrans}
\wh{Z}_a = Z_a - \Up_a.
\end{equation}

Previous observation give rise to curvature obstructions for existence of solutions of $E_0$.
If $\si$ is a solution then $\na_a \si + Z_a\si=0$. If we differentiate this relation once more, project
either to symmetric or skew-symmetric part and use $E_0(\si)=0$ then 
\begin{equation} \label{obstruct}
(dZ)_{ab} =0 \quad \text{and}  \quad \Ph(Z)_{ab}=0
\end{equation}
where
\begin{align*}
& (dZ)_{ab} = 2\na_{[a}Z_{b]} \in \cE_{[ab]},  \\
& \Ph(Z)_{ab} := \na_{(a}Z_{b)_0} - \Rho_{(ab)_0} - Z_{(a}Z_{b)_0}  \in \cE_{(ab)_0}.
\end{align*}
In fact, one can easily see that $E_0$ has a solution if and only if~\eqref{obstruct} holds.
In particular, the genericity assumption means the solution space of $E_0$
is at most 1-dimensional.  Finally note that both $(dZ)_{ab}$ and  $\Ph(Z)_{ab}$ are conformally invariant
for the scale-dependent 1-form $Z_a$ satisfying~\eqref{Ztrans}. This is easy to verify
 by a direct computation. We refer to~\cite{GoNu} for a more detailed presentation.

\subsection{Invariant differential operators} \label{s.invariant}
Invariant differential operators which are also natural -- 
in the sense they depend only on the conformal class $[g]$ and not other choices -- are fully classified 
in the locally flat setting~\cite{BC}. Existence of their
curved analogous is more involved but clearly operators $E_i$ from Section~\ref{s.BGG} are well defined
on any conformal structure. 

We shall need more operators which will be defined using the 1-form 
$Z$ from~\eqref{compinv}. These operators will be conformally invariant but not natural as they are defined
only for generic conformal classes and also depend on the choice on inversion $\ol{W}{}^{abcd}$. 
First observe that
\begin{equation} \label{newna}
\wt{\na}_a: \cE[1] \to \cE_a[1], \quad \wt{\na}_a\si = \na_a \si + Z_a \si
\end{equation}
for $\si \in \cE[1]$, is a conformally invariant connection on $\cE[1]$. 
(Note that $\wt{\na}$ is called the $\mathcal{C}$-connection in~\cite{cEspaces}.)
Indeed,  it follows from~\eqref{na-trans} that $\na_a\si$ transforms
to $\na_a \si + \Up_a \si$ in another metric $e^{2\Up}g_{ab}$, $\Up_a = \na_a \Up$ and $Z_a$
transforms to $Z_a - \Up_a$ according to~\eqref{Ztrans}.
Note the curvature of $\wt{\na}_a$ is $(dZ)_{ab} = 2\na_{[a}Z_{b]}$. Further, we shall denote 
exterior derivative twisted by $\wt{\na}_a$ by $\newd := d^{\wt{\na}}$,
\begin{equation} \label{newd}
\begin{split}
\newd: \cE_{[a_1 \ldots a_k]}[1] &\to \cE_{[a_0 \ldots a_k]}[1], \\
\tau_{a_1 \ldots a_k} &\mapsto
(k+1) \bigl( \na_{[a_0} \tau_{a_1 \ldots a_k]} + Z_{[a_0}  \tau_{a_1 \ldots a_k]} \bigr)
\end{split}
\end{equation}
for $\tau_{a_1 \ldots a_k} \in \cE_{[a_1 \ldots a_k]}[1]$ and $k \geq 0$.
By construction,  $\newd$ is conformally invariant.

Next, we shall need the operator 
\begin{equation} \label{D1}
\begin{split}
D_1: \cE_b[1] \to \cE_{(ab)_0}[1], \quad
\tau_b \mapsto \na_{(a} \tau_{b)_0} - Z_{(a}  \tau_{b)_0}
\end{split}
\end{equation}
for $\ta_b \in \cE_b[1]$.
Using a similar computation as above, one easily verifies conformal invariance of $D_1$. 
The role of the subscript in this notation will be clarified later.
Further, we denote by $C_1$ the composition of $E_1$ with the inversion $\ol{W}{}^{abcd}$,
\begin{equation} \label{C1}
\begin{split}
C_1: \cE_{(ab)_0}[1] \to \cE_{a}[1], \quad
\tau_{ab} \mapsto 2\ol{W}{}^{rst}{}_a \na_{[r} \tau_{s]t} 
\end{split}
\end{equation}
for $\tau_{ab} \in \cE_{(ab)_0}[1]$. This operator is invariant by construction. 

Operators $D_1$ and $C_1$ interact nicely with $E_0$ and $\newd$. Specifically
\begin{equation} \label{comm}
\begin{split}
D_1 \circ \newd &= E_0 + \Ph(Z): \cE[1] \to \cE_{(ab)_0}[1], \\
C_1 \circ E_0 &= \newd: \cE[1] \to \cE_a[1] ,
\end{split}
\end{equation}
where we consider $\Ph(Z)$ as an algebraic operator.
This follows from~\eqref{compinv}, \eqref{obstruct} and a direct computation. Using these
operators, we can also recover obstructions from~\eqref{obstruct},
\begin{equation} \label{obstructE}
\begin{split}
&(\newd \circ  C_1 \circ E_0) (\si) = (dZ)\si, \\
&(D_1 \circ C_1 \circ E_0) (\si) = \Ph(Z)\si + E_0(\si).
\end{split}
\end{equation}

We can also compose operators $C_1$ and $D_1$. It turns out this leads
to an auxiliary operator that can be considered as an analogue of $D_1$ where the source bundle is 
replaced by $\cE_{[bc]}[1]$ and the target bundle by $\cE_a \boxtimes \cE_{bc}[1] \subseteq \cE_{a[bc]}[1]$.
Recall $\boxtimes$ is the Cartan product, i.e.,\ $\cE_a \boxtimes \cE_{bc}[1]$ obeys symmetries of the Young diagram
with one column of the length $2$ and the second column of the length
$1$, and is traceless with respect to $\bg_{ab}$. This auxiliary operator is
\begin{equation} \label{Dbar}
\begin{split}
\ol{D}: \cE_{[bc]}[1] \to\cE_{a} \boxtimes \cE_{[bc]}[1], \quad
\tau_{bc} \mapsto \Pr{}_\boxtimes ( \na_a  - 2Z_a) \ta_{bc}
\end{split}
\end{equation}
for $\tau_{bc} \in \cE_{[bc]}[1]$. Here we use the notation $\Pr_\boxtimes$ in the same way as in 
Section~\ref{s.BGG}. A standard computation verifies $\ol{D}$ is conformally invariant.

Now a direct computation verifies that
\begin{align}
\notag
C_1 \!\circ\! D_1: \cE_a[1] &\to \cE_a[1], \\ \label{CDcomp}
\tau_a &\mapsto \tau_a - (H'_1 \newd\tau)_a - \ol{W}{}^{rsp}{}_a
\bigl[ 2 \Ph(Z)_{pr} \tau_s + \tfrac12 (dZ)_{rs} \tau_p \bigr] , \\
\notag
\text{with} ~
H'_1: \cE_{[ab]}[1] &\to \cE_a[1] , \\ \label{H'_1}
\psi_{ab} &\mapsto -\tfrac12 \ol{W}{}^{rsp}{}_a \ol{D}(\psi)_{prs} .
\end{align}
Conformal invariance of $\ol{D}$ means that also $H'_1$ is conformally invariant.

Finally note that all invariant operators $E_i$, $\wt{\na}$, $\newd$, $C_1$, $D_1$, $\ol{D}$, $H'_1$
are operators between bundles with all ``indices downstairs'' with the weight 1.
That is, these operators do not change the conformal weight.

\section{Compatibility operators} \label{sec:compat}

Below, in Section~\ref{sec:comp-background} we briefly recall the
precise definition of a compatibility complex
(Definition~\ref{def:compat}). Then, in
Section~\ref{sec:comp-construction} we recall the method given
in~\cite{kh-compat} of constructing a compatibility complex for a
differential operator that is uniformly of finite type
(Propositions~\ref{prp:compat-sufficient} and~\ref{prp:lift-compat}).
This method will later be used in the desired construction of a
compatibility complex for the conformal-to-Einstein operator
$E_0$~\eqref{E} (Section~\ref{sec:c2e-compat}) for `generic' conformal
structures. Finally, in Section~\ref{s.cc-flat} we sketch how this
method can also establish that a BGG sequence (such as~\eqref{BGGseq})
on a locally flat geometry constitutes a compatibility complex in the
sense defined here, which is an expected result that is sometimes taken
for granted, but is difficult to locate in the literature.

\subsection{Background} \label{sec:comp-background}

We start by introducing some basic notions from \emph{homological
algebra}~\cite{weibel}.

\begin{definition} \label{def:homalg}
A (possibly infinite) composable sequence $K_l$ of linear maps,
$l=l_{\min}, \ldots, l_{\max}$, such that
$K_{l+1} \circ K_l = 0$ when possible, is called a \emph{(cochain)
complex}. Given complexes $K_l$ and $K'_l$ a sequence $C_l$ of linear
maps, as in the diagram
\begin{equation}
\begin{tikzcd}[column sep=large,row sep=large]
	\cdots \ar{r} \&
	\bullet \ar{r}{K_{l-1}} \ar{d}{C_{l-1}} \&
	\bullet \ar{r}{K_l} \ar{d}{C_l} \&
	\bullet \ar{r}{K_{l+1}} \ar{d}{C_{l+1}} \&
	\bullet \ar{r} \ar{d}{C_{l+2}} \&
	\cdots
	\\
	\cdots \ar{r} \&
	\bullet \ar[swap]{r}{K'_{l-1}} \&
	\bullet \ar[swap]{r}{K'_l} \&
	\bullet \ar[swap]{r}{K'_{l+1}} \&
	\bullet \ar{r} \&
	\cdots
\end{tikzcd} ,
\end{equation}
such that its squares commute, that is $K'_l \circ C_l = C_{l+1} \circ
K_l$ when possible, is called a \emph{cochain map} or a \emph{morphism}
between complexes. A \emph{homotopy} between complexes $K_l$ and $K'_l$
(which could also be the same complex, $K_l = K'_l$) is a sequence of
morphism, as the dashed arrows in the diagram
\begin{equation}
\begin{tikzcd}[column sep=large,row sep=large]
	\cdots \ar{r} \&
	\bullet \ar{r}{K_{l-1}} \ar{d}{C_{l-1}} \&
	\bullet \ar{r}{K_l} \ar{d}{C_l} \ar[dashed]{dl}{H_{l-1}} \&
	\bullet \ar{r}{K_{l+1}} \ar{d}{C_{l+1}} \ar[dashed]{dl}{H_l} \&
	\bullet \ar{r} \ar{d}{C_{l+2}} \ar[dashed]{dl}{H_{l+1}}\&
	\cdots
	\\
	\cdots \ar{r} \&
	\bullet \ar[swap]{r}{K'_{l-1}} \&
	\bullet \ar[swap]{r}{K'_l} \&
	\bullet \ar[swap]{r}{K'_{l+1}} \&
	\bullet \ar{r} \&
	\cdots
\end{tikzcd} ,
\end{equation}
and the sequence of maps $C_l = K'_{l-1}\circ H_{l-1} + H_l\circ K_l$ is
said to be a \emph{morphism induced by} the homotopy $H_l$. An
\emph{equivalence up to homotopy} between complexes $K_l$ and $K'_l$ is
a pair of morphisms $C_l$ and $D_l$ between them, as in the diagram
\begin{equation}
\begin{tikzcd}[column sep=large,row sep=4.5em]
	\ar[loop left]{}{\tH_{l_{\min}-1}}
	\bullet \ar{r}{K_{l_{\min}}}
		\ar[swap,shift right]{d}{C_{l_{\min}}} \&
	\ar[r,phantom,"\cdots"]
		\ar[dashed,bend left]{l}{H_{l_{\min}}} \&
	\bullet \ar{r}{K_l}
		\ar[swap, shift right]{d}{C_l} \&
	\bullet \ar[r,phantom,"\cdots"]
		\ar[swap,shift right]{d}{C_{l+1}}
		\ar[dashed,bend left]{l}{H_l} \&
	\ar{r}{K_{l_{\max}}} \&
	\bullet \ar[swap,shift right]{d}{C_{l_{\max}+1}}
		\ar[dashed,bend left]{l}{H_{l_{\max}}}
		\ar[loop right]{}{\tH_{l_{\max}+1}}
	\\
	\ar[loop left]{}{\tH'_{l_{\min}-1}}
	\bullet \ar[swap]{r}{K'_{l_{\min}}}
		\ar[swap,shift right]{u}{D_{l_{\min}}} \&
	\ar[r,phantom,"\cdots"]
		\ar[swap,dashed,bend right]{l}{H'_{l_{\min}}} \&
	\bullet \ar[swap]{r}{K'_l}
		\ar[swap,shift right]{u}{D_l} \&
	\bullet \ar[r,phantom,"\cdots"]
		\ar[swap,shift right]{u}{D_{l+1}}
		\ar[swap,dashed,bend right]{l}{H'_l} \&
	\ar[swap]{r}{K'_{l_{\max}}} \&
	\bullet \ar[swap,shift right]{u}{D_{l_{\max}+1}}
		\ar[swap,dashed,bend right]{l}{H'_{l_{\max}}}
	\ar[loop right]{}{\tH'_{l_{\max}+1}}
\end{tikzcd} ,
\end{equation}
such that $C_l$ and $D_l$ are mutual inverses up to homotopy ($H_l$ and
$H'_l$), that is
\begin{equation} \label{eq:CD-equiv}
\begin{aligned}
	D_l \circ C_l &= \id - K_{l-1} \circ H_{l-1} - H_l \circ K_l ,
	\\
	C_l \circ D_l &= \id - K'_{l-1} \circ H'_{l-1} - H'_l \circ K'_l ,
\end{aligned}
\end{equation}
with the special end cases
\begin{equation} \label{eq:CD-equiv-edge}
\begin{aligned}
	D_{l_{\min}} \circ C_{l_{\min}}
	&= \id - \tH_{l_{\min}-1} - H_{l_{\min}} \circ K_{l_{\min}} , &
	K_{l_{\min}} \circ \tH_{l_{\min}-1} &= 0 ,
	\\
	C_{l_{\min}} \circ D_{l_{\min}}
	&= \id - \tH'_{l_{\min}-1} - H'_{l_{\min}} \circ K'_{l_{\min}} , &
	K'_{l_{\min}} \circ \tH'_{l_{\min}-1} &= 0 ,
	\\
	D_{l_{\max}+1} \circ C_{l_{\max}+1}
	&= \id - H_{l_{\max}} \circ K_{l_{\max}} - \tH_{l_{\max}+1} , &
	\tH_{l_{\max}+1} \circ K_{l_{\max}}  &= 0 ,
	\\
	C_{l_{\max}+1} \circ D_{l_{\max}+1}
	&= \id - H'_{l_{\max}} \circ K'_{l_{\max}} - \tH'_{l_{\max}+1} , &
	\tH'_{l_{\max}+1} \circ K'_{l_{\max}}  &= 0 ,
\end{aligned}
\end{equation}
where the $\tH$ maps are allowed to be arbitrary, as long as they
satisfy the given identities. The $\tH$ operators may be omitted from
the data, since they can be defined to satisfy to the left column
identities in~\eqref{eq:CD-equiv-edge}, as long as they satisfy the
identities in the left column of~\eqref{eq:CD-equiv-edge}.
\end{definition}

Note that our definition of equivalence up to homotopy between complexes
of finite length is set up in a way that allows an equivalence between
longer complexes to be truncated and still remain an equivalence.

Next, we restrict our attention to the case where all maps are given by
differential operators.

\begin{definition}[{cf.~\cite[Def.10.5.4]{seiler-inv}, \cite[Defs.1.2.2--4]{tarkhanov}}] \label{def:compat}
Given a differential operator $K$, any composable differential operator
$L$ such that $L\circ K = 0$ is a \emph{(partial) compatibility
operator} for $K$. If $K_1$ is a compatibility operator for $K$, it is
called \emph{complete} or \emph{universal} when any other compatibility
operator $L$ can be factored through $L = L'\circ K_1$ for some
differential operator $L'$. A complex of differential operators $K_l$,
$l=0,1,\ldots$ is called a compatibility complex for $K$ when $K_0 = K$
and, for each $l\ge 1$, $K_l$ is a complete compatibility operator for
$K_{l-1}$.
\end{definition}

Our usage of attributes like \emph{partial}, \emph{complete} or
\emph{universal} is not standard. In the literature, when no attribute
is used, the completeness condition is
implied~\cite[Def.1.2.2]{tarkhanov}, while a partial compatibility
operator does not seem to have a standard name.

\subsection{Sufficient conditions and construction} \label{sec:comp-construction}

We can make two related remarks about the properties of compatibility operators.
A complete compatibility operator, say $K_1$, need not be
unique. But, by its universal factorization property, any two
compatibility operators, say $K_1$ and $K'_1$, must factor through each
other, $K_1 = L_1 \circ K'_1$ and $K'_1 = L'_1\circ K_1$ for some
differential operators $L_1$ and $L'_1$. On the other hand, while the
compatibility condition $K_1 \circ K = 0$ is easy to check, it may be
quite challenging to verify the completeness/universality property of
any particular $K_1$. However, of both $K$ and $K_1$ may be reduced to
some other pair of operators, say $K'$ and $K'_1$, where the 
compatibility and completeness conditions are known to be satisfied,
this reduction may serve as a witness to the completeness property of
$K_1$. Both of these remarks motivate the utility of the following

\begin{proposition}[{\cite[Lem.4]{kh-compat}}] \label{prp:compat-sufficient}
Consider two complexes of differential operators $\Kop_l$ and $\Kop_l'$, for
$l=0,1,\cdots,n-1$, and an equivalence up to homotopy between them, as
in the diagram
\begin{equation} \label{cc-equiv}
\begin{tikzcd}[column sep=2cm,row sep=2cm]
	\bullet \ar{r}{\Kop_0}
		\ar[swap,shift right]{d}{\Cop_0} \&
		\ar[dashed,bend left]{l}{\Hop_0}
	\bullet \ar{r}{\Kop_1}
		\ar[swap, shift right]{d}{\Cop_1} \&
	\bullet \ar[r,phantom,"\cdots"]
		\ar[swap,shift right]{d}{\Cop_2}
		\ar[dashed,bend left]{l}{\Hop_1} \&
	\bullet
		\ar[swap,shift right]{d}{\Cop_{n-1}}
		\ar{r}{\Kop_{n-1}} \&
	\bullet \ar[swap,shift right]{d}{\Cop_n}
		\ar[dashed,bend left]{l}{\Hop_{n-1}}
	\\
	\bullet \ar[swap]{r}{\Kop'_0}
		\ar[swap,shift right]{u}{\Dop_0} \&
		\ar[swap,dashed,bend right]{l}{\Hop'_0}
	\bullet \ar[swap]{r}{\Kop'_1}
		\ar[swap,shift right]{u}{\Dop_1} \&
	\bullet \ar[r,phantom,"\cdots"]
		\ar[swap,shift right]{u}{\Dop_2}
		\ar[swap,dashed,bend right]{l}{\Hop'_1} \&
	\bullet
		\ar[swap,shift right]{u}{\Dop_{n-1}}
		\ar[swap]{r}{\Kop'_{n-1}} \&
	\bullet \ar[swap,shift right]{u}{\Dop_n}
		\ar[swap,dashed,bend right]{l}{\Hop'_{n-1}}
\end{tikzcd} ,
\end{equation}
where for simplicity we are assuming that the $\tilde{\Hop}_{-1}$,
$\tilde{\Hop}'_{-1}$, $\tilde{\Hop}_n$, $\tilde{\Hop}'_n$ are all zero.
If $\Kop'_l$ is a compatibility complex for $\Kop'_0$, then $\Kop_l$ is
a compatibility complex for $\Kop_0$. 
\end{proposition}

The observation that is at the core of our method of explicitly
constructing compatibility operators, together with a witness to their
completeness property is the following

\begin{proposition}[{cf.~\cite[Lem.5]{kh-compat}, \cite[Prp.1.2.7]{tarkhanov}}] \label{prp:lift-compat}
Consider the following diagram of differential operators:
\begin{equation} \label{eq:lifted-compat}
\begin{tikzcd}[column sep=8em,row sep=7em]
	\bullet \ar{r}{K_0}
		\ar[swap,shift right]{d}{C_0} \&
	\bullet \ar{r}{K_1
		= \begin{bmatrix}
			\id - K_0\circ H_0 - D_1\circ C_1 \\
			K'_1 \circ C_1\end{bmatrix}}
		\ar[swap,shift right]{d}{C_1}
		\ar[dashed,bend left]{l}{H_0} \&
	\bullet
		\ar[swap,shift right]{d}{C_2
			= \begin{bmatrix}
				0 & \id\end{bmatrix}}
		\ar[dashed,bend left]{l}{H_1
			= \begin{bmatrix}
				\id & 0\end{bmatrix}}
	\\
	\bullet \ar[swap]{r}{K'_0}
		\ar[swap,shift right]{u}{D_0} \&
	\bullet \ar[swap]{r}{K'_1}
		\ar[swap,shift right]{u}{D_1}
		\ar[swap,dashed,bend right]{l}{H'_0} \&
	\bullet
		\ar[swap,shift right]{u}{D_2
			= \begin{bmatrix}
				D'_2 \\ \id - K'_1\circ H'_1\end{bmatrix}}
		\ar[swap,dashed,bend right]{l}{H'_1}
\end{tikzcd} .
\end{equation}
If its truncation to the first square is an equivalence up to homotopy
in the sense of Definition~\ref{def:homalg} (of length 1) and the
operator $K'_1$ is known to be a complete compatibility operator for
$K'$, then there exist differential operators $H'_1$ and $D'_2$ such
that the second square in the above diagram completes it to an
equivalence up to homotopy (of length 2).
\end{proposition}

A consequence of this observation is that, by
Proposition~\ref{prp:compat-sufficient} the operator $K_1$
in~\eqref{eq:lifted-compat} is witnessed to be also be a complete
compatibility operator for $K$. In practice, this $K_1$ operator can
contain quite a lot of redundant information, so some simplification by
hand may be required to bring it into a more convenient form. As long as
the simplifications keep track of the witnessing equivalence data, then
the simplified $K_1$ is still easily verified to be complete.

\subsection{BGG sequence as a compatibility complex} \label{s.cc-flat}

Here we sketch an argument for the result that
BGG complexes provide compatibility complexes 
(of first BGG operators) on locally flat parabolic geometries.
Here we sketch how this follows from Proposition~\ref{prp:compat-sufficient}. 
A version of this argument appears in~\cite[Sec.4]{celm-lorentz}, in the
specific case of the Calabi complex, the BGG sequence for the Killing
operator.
We assume references~\cite{CSS,CD} as background material for this section.

A BGG complex is parametrized by choice a tractor bundle $\cV$. Some
representation theoretic data then builds its operators, which we can
generically call $E_l$, and supplies all the data that produces an
instance of an equivalence up to homotopy, as illustrated
in~\eqref{cc-equiv}, with a reference compatibility complex. We
illustrate a generic square of the resulting diagram, which can be
referred to in the rest of the discussion:
\begin{equation} \label{bgg-equiv}
\begin{tikzcd}[column sep=3cm,row sep=3cm]
	V_l \ar{r}{E_l = \pi_{l+1}d^{\ol{\na}} \Pi_l}
		\ar[swap, shift right]{d}{\Pi_l} \&
	V_{l+1}
		\ar[swap,shift right]{d}{\Pi_{l+1}}
		\ar[dashed,bend left]{l}{0}
	\\
	V_l' \ar[swap]{r}{d^{\ol{\na}}}
		\ar[swap,shift right]{u}{\pi_l P'_l(d^{\ol{\na}} \partial^*)} \&
	V_{l+1}'
		\ar[swap,shift right]{u}{\pi_{l+1} P'_{l+1}(d^{\ol{\na}} \partial^*)}
		\ar[swap,dashed,bend right]{l}{-\partial^* Q_{l}(d^{\ol{\na}} \partial^*)}
\end{tikzcd} .
\end{equation}
The lower level of~\eqref{cc-equiv}
is given by bundles  $V_l' := \La^l T^*M \otimes \cV$ together with 
$K_l' = d^{\ol{\na}}: \Ga(V_l') \to \Ga(V_{l+1}')$  where $\ol{\na}$
is the flat tractor connection on $\cV$. 

Lie algebraic data for parabolic geometries give rise to the Kostant's codifferential and its bundle version
$\partial^*: V_l' \to V_{l-1}'$. We put  $V_l := \ker \partial^*/ \im \partial^*$ in $V_l'$. 
These are the bundles in the upper level of~\eqref{cc-equiv}.
By construction, there is a projection $\pi_l: \ker \partial^*|_{V_l'} \to V_l$ and, moreover, there is
a differential splitting $\Pi_l: \Ga(V_l) \to \Ga(\ker \partial^*|_{V_l'})$ which is uniquely
characterized by conditions $\partial^* \Pi_l=0$, $\partial^* d^{\ol{\na}} \Pi_l=0$ and 
$\pi_l \circ \Pi_l = \id_{V_l}$. The interaction of these tools gives rise to the calculus known
as ``BGG machinery''~\cite{CSS,CD}. 
In particular, this yields BGG operators $K_l = E_l$ in the upper level of~\eqref{cc-equiv} in the form 
$E_l = \pi_{l+1}d^{\ol{\na}} \Pi_l$.

The last ingredient we need is the differential operator 
$\Box = \partial^* d^{\ol{\na}} + d^{\ol{\na}} \partial^*: \Ga(V_l') \to \Ga(V_l')$.
Note that $\Box$ reduces to $\partial^* d^{\na}$ on sections of $\im \partial^*$ and to 
$d^{\ol{\na}} \partial^*$ on sections of associated graded bundle of  $V_l'/\ker \partial^*$.
Since $\Box$ is closely related~\cite{CD,CS} to the so called Kostant's Laplacian~\cite{Kostant}, 
it follows from representation theory that we have polynomials
$P_l$ and $P'_l$, both with constant term equal to $1$, such that corresponding polynomial operators satisfy
\begin{align*}
P_l(\partial^* d^{\na}) |_{\Ga(\im \partial^*)} &= 0, &
P'_l(d^{\ol{\na}} \partial^*): \Ga(V_l') &\to \Ga(\ker \partial^*), \\
P_l(\partial^* d^{\na}) P'_l(d^{\ol{\na}} \partial^*) |_{\im \Pi} &= \id, &
P'_l &= P_{l-1}.
\end{align*}

Now we can specify all missing operators in~\eqref{cc-equiv}. We put $C_l := \Pi_l$,  
$D_l := \pi_l P'_l(d^{\ol{\na}} \partial^*)$ and $H_l=0$. Further, using polynomials
$Q_l(x)$ given by 
$P_l(x) = xQ_l(x)+1$ and properties in the previous display, we have
\begin{align*}
P_l(\partial^* d^{\ol{\na}}) & P'_l(d^{\ol{\na}} \partial^*) = 
\bigl( d^{\ol{\na}} \partial^* Q_{l-1}(d^{\ol{\na}} \partial^*) + \id \bigr) 
\bigl( Q_l(\partial^* d^{\ol{\na}})\partial^* d^{\ol{\na}} + \id \bigr)   = \\
&= \id + d^{\ol{\na}} \partial^* Q_{l-1}(d^{\ol{\na}} \partial^*) 
+  \partial^* Q_l( d^{\ol{\na}} \partial^*) d^{\ol{\na}} = C_lD_l.
\end{align*}
Now we put $H'_l = -\partial^* Q_{l}(d^{\ol{\na}} \partial^*)$ and
the identities~\eqref{eq:CD-equiv} are satisfied, which verifies equivalence of the BGG complex $E_l$ and 
the twisted de-Rham complex $d^{\ol{\na}}$ are equivalent  up to homotopy, hence $E_l$ is a compatibility complex
for $E_0$.

\section{Compatibility complex for C2E operator} \label{sec:c2e-compat}

Assume the conformal structure $(M,[g_{ab}])$ of dimension $n \geq 4$ is generic in the sense of~\eqref{generic}.
We shall comment upon the dimension $3$ in Remark~\ref{lowdim}.
Then the solution space of $E(\si)=0$ has either dimension $1$ or $0$. We shall discuss these two cases separately.
Recall 1-dimensional solution space is characterized by the property~\eqref{obstruct}, see 
Section~\ref{s.generic}.

\subsection{No non-trivial solution} \label{sec:nosol}
Following~\eqref{obstruct}, this case is characterized by
\[
	\begin{pmatrix}
		(dZ)_{ab} \\
		\Ph(Z)_{ab}
	\end{pmatrix}
	\ne 0 ,
	\quad \text{equivalently} \quad
	\begin{pmatrix}
		\zeta^{ab} & \eta^{ab}
	\end{pmatrix}
	\begin{pmatrix}
		(d Z)_{ab} \\
		\Ph(Z)_{ab}
	\end{pmatrix}
	= 1 ,
\]
for some tensor fields $\zeta^{ab} \in \cE^{[ab]}$ and $\eta^{ab} \in \cE^{(ab)_0}$. 
Then there exists a left inverse $H_0: \cE_{(ab)_0}[1] \to \cE[1]$ of $E=E_0$,
\begin{equation} \label{eq:E0-invert}
\begin{aligned}
	&H_0 \circ E_0 = \id_{\cE[1]} , \quad \text{with} \\
	&H_0(\tau) :=
	\begin{pmatrix}
		\zeta^{ab} & \eta^{ab}
	\end{pmatrix}
	\begin{pmatrix}
		(\newd \circ C_1 (\tau))_{ab} \\
		((D_1 \circ C_1 - \id) (\tau))_{ab}
	\end{pmatrix}
\end{aligned}
\end{equation}
for $\tau_{ab} \in \cE_{(ab)_0}[1]$, cf.~\eqref{obstructE}.

If we blindly apply the successive construction of the compatibility
operators from Proposition~\ref{prp:lift-compat}, we end up with the
following equivalence diagram, which repeats the last two squares
indefinitely and hence has infinite length:
\begin{equation} \label{eq:generic-nosol-equiv-inf}
\begin{tikzcd}[column sep=3cm,row sep=3cm]
	\bullet \ar{r}{E_0}
		\ar[swap,shift right]{d}{0} \&
	\bullet \ar{r}{\id_{\cE_{(ab)_0}[1]} - E_0 H_0}
		\ar[swap,shift right]{d}{0}
		\ar[dashed,bend left]{l}{H_0}
		\&
	\bullet \ar{r}{E_0 H_0}
		\ar[swap,shift right]{d}{0}
		\ar[dashed,bend left]{l}{\id_{\cE_{(ab)_0}[1]}}
		\&
	\bullet \rlap{ $\cdots$}
		\ar[swap,shift right]{d}{0}
		\ar[dashed,bend left]{l}{\id_{\cE_{(ab)_0}[1]}}
	\\
	0 \ar[swap]{r}{0}
		\ar[swap,shift right]{u}{0} \&
	0 \ar[swap]{r}{0}
		\ar[swap,shift right]{u}{0}
		\ar[swap,dashed,bend right]{l}{0}
		\&
	0 \ar[swap]{r}{0}
		\ar[swap,shift right]{u}{0}
		\ar[swap,dashed,bend right]{l}{0}
		\&
	0  \rlap{ $\cdots$}
		\ar[swap,shift right]{u}{0}
		\ar[swap,dashed,bend right]{l}{0}
\end{tikzcd}
\end{equation}
However, it is easy to notice that the second compatibility operator
$E_0 H_0$ contains many redundant components, since by
using~\eqref{eq:E0-invert} we can factor
\[
	H_0 = H_0 (E_0 H_0) = (H_0 E_0) H_0.
\]
It is straightforward to rewrite the third equivalence square
in~\eqref{eq:generic-nosol-equiv-inf} using $H_0$ instead of $E_0 H_0$.
Then, continuing to apply the construction from
Proposition~\eqref{prp:lift-compat} we arrive at a much simplified
equivalence diagram, now of length 3:
\begin{equation} \label{eq:generic-nosol-equiv}
\begin{tikzcd}[column sep=3cm,row sep=3cm]
	\bullet \ar{r}{E_0}
		\ar[swap,shift right]{d}{0} \&
	\bullet \ar{r}{\id - E_0 H_0}
		\ar[swap,shift right]{d}{0}
		\ar[dashed,bend left]{l}{H_0}
		\&
	\bullet \ar{r}{H_0}
		\ar[swap,shift right]{d}{0}
		\ar[dashed,bend left]{l}{\id}
		\&
	\bullet
		\ar[swap,shift right]{d}{0}
		\ar[dashed,bend left]{l}{E_0}
	\\
	0 \ar[swap]{r}{0}
		\ar[swap,shift right]{u}{0} \&
	0 \ar[swap]{r}{0}
		\ar[swap,shift right]{u}{0}
		\ar[swap,dashed,bend right]{l}{0}
		\&
	0 \ar[swap]{r}{0}
		\ar[swap,shift right]{u}{0}
		\ar[swap,dashed,bend right]{l}{0}
		\&
	0
		\ar[swap,shift right]{u}{0}
		\ar[swap,dashed,bend right]{l}{0}
\end{tikzcd}
\end{equation}
We can encapsulate the result of this calculation in the following rather generic
\begin{lemma} \label{lem:nosol-equiv}
Given a differential operator $E_0$, provided there exists a
left-inverse differential operator $H_0$, such that
\begin{equation} \label{eq:H0-nosol}
	H_0 \circ E_0 = \id ,
\end{equation}
then the top line of~\eqref{eq:generic-nosol-equiv} is a full
compatibility complex for $E_0$ of length 3.
\end{lemma}
\begin{proof}
The identity~\eqref{eq:H0-nosol} is sufficient to verify that the data
in diagram~\eqref{eq:generic-nosol-equiv} constitutes an equivalence of
complexes up to homotopy, with the bottom line trivially being a
compatibility complex. Therefore, by
Proposition~\ref{prp:compat-sufficient}, the top line is also a
compatibility complex. By inspection, the construction of
Proposition~\ref{prp:lift-compat} yields $0$ as the next compatibility
operator. Hence the diagram can be truncated at length 3.
\end{proof}

\subsection{One dimensional solution space} \label{sec:onesol}
This case is characterized by $dZ=0$ and $\Ph(Z)=0$, see~\eqref{obstruct} and discussion below.
We assume $n \geq 4$ as a one dimensional solution space is not possible for $n=3$. In the latter dimension, 
existence of an Einstein metric in $[g_{ab}]$ means the Cotton tensor is zero hence $[g_{ab}]$ is conformally flat.

It turns out, we shall need a rather specific case of Proposition~\ref{prp:lift-compat} here: 
it is sufficient to assume $C_0=D_0=\id$ and $H_0 = H_0'=0$ in the first square of~\eqref{eq:lifted-compat}. Then
one can follow Proposition~\ref{prp:lift-compat} to obtain a compatibility complex. However -- similar to
Section~\ref{sec:nosol} -- this can be heavily simplified. Using  a small number of key identities, we obtain 
the following general result:

\begin{lemma} \label{lem:onesol-equiv}
Consider the following diagram of differential operators:
\begin{equation} \label{eq:onesol-equiv}
\begin{tikzcd}[column sep=2.cm,row sep=3cm]
	V_0 \ar{r}{K_0}
		\ar[swap,shift right]{d}{\id}
		\&
	V_1 \ar[dashed,bend left]{l}{0}
		\ar[swap,shift right]{d}{C_1}
		\ar{r}{\left( \begin{smallmatrix} \id - D_1 C_1 \\ K'_1 C_1 \end{smallmatrix} \right)}
		\&
	\begin{smallmatrix} V_1 \\ \oplus \\ V'_2 \end{smallmatrix}
		\ar{r}{\left(\begin{smallmatrix} C_1 & -H'_1 \\ 0 & K'_2 \end{smallmatrix}\right)}
		\ar[swap,shift right]{d}{\left(\begin{smallmatrix} 0 & \id \end{smallmatrix}\right)}
		\ar[dashed,bend left]{l}{\left(\begin{smallmatrix} \id & 0 \end{smallmatrix} \right)}
		\&	
	\begin{smallmatrix} V'_1 \\ \oplus \\ V'_3 \end{smallmatrix}
		\ar{r}{\left(\begin{smallmatrix} 0 & K'_3 \end{smallmatrix}\right)}
		\ar[dashed,bend left]{l}{\left(\begin{smallmatrix} D_1 & 0 \\ -K'_1 & 0 \end{smallmatrix} \right)}
		\ar[swap,shift right,pos=0.7]{d}{\left(\begin{smallmatrix} 0 & \id \end{smallmatrix}\right)}
		\&
	V'_4 ~ \rlap{$\cdots$}
		\ar[dashed,bend left]{l}{0}
		\ar[swap,shift right]{d}{\id}
	\\
	V_0 \ar[swap,shift right]{r}{K'_0}
		\ar[swap,shift right]{u}{\id} 
		\&
	V'_1 \ar[swap,shift right]{r}{K'_1}
		\ar[swap,shift right]{u}{D_1}
		\ar[swap,dashed,bend right]{l}{0}
		\&
	V'_2 \ar[swap,shift right]{r}{K'_2}
		\ar[swap,dashed,bend right]{l}{H'_1}
		\ar[swap,shift right]{u}{\left( \begin{smallmatrix} D_1 H'_1 \\ \id - K'_1 H'_1 \end{smallmatrix} \right)}
		\&
	V'_3 \ar[swap,shift right]{r}{K'_2}
		\ar[swap,shift right,pos=0.3]{u}{\left(\begin{smallmatrix} 0 \\ \id \end{smallmatrix}\right)} 
		\ar[swap,dashed,bend right]{l}{0}
		\&
	V'_4 ~ \rlap{$\cdots$}
		\ar[swap,shift right]{u}{\id} 
		\ar[swap,dashed,bend right]{l}{0}
\end{tikzcd}
\end{equation}
continued to the right by $C_i = D_i = \id$ and $H_i = H'_i = 0$.
Provided that the bottom row is a full compatibility complex, its
truncation to the first cell is an equivalence up to homotopy and
moreover $H'_1$ is an operator satisfying the identity
\begin{equation} \label{eq:onesol-H1'}
	\id - C_1 \circ D_1 = H'_1 \circ K_1' ,
\end{equation}
then the diagram in~\eqref{eq:onesol-equiv} is an equivalence up to
homotopy and its top row is a full compatibility complex.
\end{lemma}
\begin{proof}
All the required identities can be checked by explicit calculation,
reducing them to the identities in the hypotheses. The conclusion
follows from a simple application of Proposition~\ref{prp:lift-compat}.
\end{proof}

An immediate consequence is our main Theorem~\ref{main}, which exhibits
a compatibility complex for the conformal-to-Einstein operator
$E_0$~\eqref{E} with one independent solution, with a generic Weyl tensor.

\begin{proof}[Proof of Theorem~\ref{main}.]
The complex from the statement of the theorem is exactly the top line
of~\eqref{eq:onesol-equiv}, provided that
\begin{align*}
	V_0 &= \cE[1] , &
	V_1 &= \cE_{(ab)_0}[1], &
	V'_i &= \cE_{[a_1\cdots a_i]} [1] , \\
	K_0 &= E_0, &
	K'_i &= \newd ,
\end{align*}
and operators $C_1$ and $D_1$ are given by~\eqref{C1} and~\eqref{D1},
with $H'_1$ given by~\eqref{H'_1}. So the conclusion that it is a
compatibility complex follows from Lemma~\ref{lem:onesol-equiv}. We need only
check the commutativity of the first square of~\eqref{eq:onesol-equiv},
which is ensured by~\eqref{comm}, and the identity~\eqref{eq:onesol-H1'}
involving $H'_1$, which is just a rearrangement of
identity~\eqref{CDcomp}.
\end{proof}

\begin{remark} \label{lowdim}
It turns out that our construction of compatibility complexes provides a complete answer 
in certain low dimensional cases.

(1) Assume dimension $n=3$. Then the Weyl tensor is identically zero for all conformal structures $(M,[g])$.
We need to use the Cotton tensor $Y_{abc}$ instead:  $Y_{abc}$ vanishes if and only if $(M,[g])$ 
is conformally flat. Moreover, the solution space of $E_0$ is trivial unless $(M,[g])$ is flat.
Thus the suitable notion of genericity is the assumption $Y_{abc} \not= 0$  everywhere.
Then using $\zeta^{abc} Y_{abc}=1$ for some $\zeta^{abc}$, we have $\zeta^{abc} (E_1 \circ E_0(\si))_{abc} = \si$
according to~\eqref{comp}. This gives rise to $H_0$ satisfying $H_0 \circ E_0 = \id_{\cE[1]}$ 
and one can follow Section~\ref{sec:nosol} to construct a compatibility complex (of the length 3).

(2) Assuming  $n=4$, there is  the well-known dimension-dependent identity~\cite[Eq.(31)]{edgar-hoglund}
\begin{equation} \label{4dim}
W^{abcd}W_{abce} = \tfrac14 |W|^2 \delta_d^e. \quad |W|^2 = W^{abcd}W_{abcd}.
\end{equation}
Thus assuming moreover the Riemannian signature, the  invertibility discussion at~\eqref{eq:V-from-W} 
and~\eqref{eq:Winv-from-V} shows that the Weyl tensor
is either generic or identically zero. Further, one easily verifies that two linearly independent solutions
$\si, \bar{\si} \in \cE[1]$ of $E_0$ give rise to the conformal Killing field 
$k_a = \si \na_a \bar{\si} - \bar{\si} \na_a \si$ such that $W_{abcd}k^d=0$ (in any dimension and signature),
cf.~\cite{Glap};
this means non-genericity -- therefore conformal flatness ($W=0$) -- in dimension 4. Summarizing, 
we have the compatibility complex for all 4-dimensional conformal structures in the Riemannian signature.

(3) One can consider similar identities as in~\eqref{4dim} with more complicated tensor fields.
Assume $n=4$ and consider the following polynomial tensor in $W_{abcd}$,
$$
\overset{2}{W}_{abcd} := W_{rsab} W^{rs}{}_{cd} \quad \text{and} \quad
\overset{3}{V}{}_a^b := W_{rs}{}^{tb}   W^{rs}{}_{pq} W^{pq}{}_{ta}.
$$
If $|W|^2=0$ then also $W_{rsta}W^{rstb}=0$ according to~\eqref{4dim} hence
$\overset{2}{W}_{abcd}$ is totally trace-free. Following the approach of~\cite{edgar-hoglund},
this means the vanishing of $\overset{2}{W}_{[ab}{}^{[cd}\delta_{e]}{}^{f]}=0$. Contracting this relation
with $W^{ab}{}_{cd}$ yields
$$
\overset{3}{V}{}_a^b = \tfrac14 \overset{3}{V}{}_r^r \delta_a^b.
$$
For non-definite signatures, this can be used instead of~\eqref{4dim} in cases when  $|W|^2=0$ but 
$\overset{3}{V}{}_r^r$ is nonzero everywhere.

Neither of inversions of the Weyl tensor mentioned above works for all generic 4-manifolds. 
In Appendix~\ref{app:lorentzian} we exhibit a
family of examples in Lorentzian signature in 4-dimensions where \emph{any} 
scalar covariantly formed from the Weyl tensor must vanish. These are 
Petrov Type~III Weyl tensors in the terminology of Appendix~\ref{app:lorentzian}.
\end{remark}

\section{A projective analogue} \label{sec:p2e-compat}

There is a close analogue of the conformal operator $E_0$ in projective geometry. 
In fact, also the compatibility complex has a close projective relative. To show that, 
we need to find projective versions of invariant operators discussed in Section~\ref{s.c2e};
the construction in Section~\ref{sec:c2e-compat} will be just the same (using projective
versions of all invariant operators needed there).

Here we follow the notation from~\cite{BEG}. 
The projective structure $(M,[\na])$ on a smooth orientable manifold $M$, $\dim M \geq 2$ is given by a class $\na$ of 
special projectively equivalent torsion free connections. That is, connections in $[\na]$ 
are volume preserving and have 
the same class of geodesics (as unparametrized curves).  
Then the difference between $\na, \wh{\na} \in [\na]$ is controlled
by an exact 1-form $\Up_a = \na_a \Up$, $\Up \in C^\infty(M)$,
\begin{equation} \label{na-trans-proj}
\begin{split}
& \hat{\na}_a \si = \na_a \si + w \Up_a \si, \\
& \hat{\na}_a \ta_b = \na_a \ta_b - \Up_a \ta_b - \Up_b \ta_a, \\
& \hat{\na}_a \ph^b = \na_a \ph^b + \Up_a \ph^b + \Up_r\ph^r  \delta_a^b
\end{split}
\end{equation}
where $\si \in \cE[w]$, $\ta_b \in \cE_b$ and $\ph^b \in \cE^b$. Here $\cE[w]$ are projective density
bundles, see~\cite{BEG} for details. In this convention we have $\La^n T^*M \cong \cE[-(n+1)]$.
The Ricci tensor $Ric_{ab}$ and projective Schouten tensor $\Rho_{ab}$ are related by 
$Ric_{ab} = (n-1) \Rho_{ab}$ and these are symmetric tensors (for special connections $\na$).
The projective Cotton tensor is given by $Y_{abc} = 2\na_{[b} \Rho_{c]a}$.
The projectively invariant Weyl tensor $W_{ab}{}^c{}_d$ is the trace free part of the curvature tensor
$R_{ab}{}^c{}_d$ of $\na$. Projective flatness is characterized by $W_{ab}{}^c{}_d=0$ for $n \geq 3$ and
by $Y_{abc}=0$ for $n=2$.

\vspace{1ex}

We start with the sequence of operators $E_i$ given by formula~\eqref{BGGseq} 
which we consider in the projective sense: $\cE[1]$ are projective densities and the meaning
of the Cartan product is slightly different as $\boxtimes$ denotes only suitable symmetrizations of indices 
(there is no trace to be removed when all indices are down).
This sequence of operators is projectively invariant and 
it is a complex for projective flat class $[\na]$. Following the discussion in Section~\ref{s.cc-flat}, it 
provides a compatibility complex of the \emph{projective-to-Ricci-flat} operator%
\begin{equation} \label{E-proj}
	E_0 \colon \cE[1] \to \cE_{(ab)} , \quad
	\si \mapsto (\na_a\na_b + \Rho_{ab})\si .
\end{equation}
That is, solutions $\si$ of $E_0$ without zeros correspond to Ricci-flat connections $\wh{\na}$ in $[\na]$. 
Specifically, if $\na$ corresponds to the volume form $\epsilon$ (given uniquely up to a constant multiple), 
i.e.\ $\na \epsilon =0$, and we identify $\si$ with a function using $\epsilon$, then 
$\hat{\epsilon} = \si^{-(n+1)} \epsilon$ corresponds to $\wh{\na}$.

Next we discuss a suitable genericity condition for our purpose. 
Assuming projective meaning, we have  
\begin{equation} \label{comp-proj}
F_{abc} := (E_1 \circ E_0 (\si))_{abd} = -W_{ab}{}^c{}_d(\na_c \si) +  Y_{dab} \si
	\in \bigl( \cE_{[ab]} \boxtimes \cE_{c}\bigr)[1], 
\end{equation}
differs from~\eqref{comp} by position of indices of the Weyl tensor (which we cannot raise or lower).
Thus we shall call $(M,[\na])$ generic if the bundle map $\ta_c \mapsto W_{ab}{}^c{}_d\ta_c$ is injective.
Then there is an inverse $\ol{W}{}^{ab}{}_e{}^d$ of this map, i.e.,\
$\ol{W}{}^{ab}{}_e{}^d W_{ab}{}^c{}_d = \delta_e^c$.
Applying the inverse $-\ol{W}{}^{ab}{}_e{}^d$ to the previous display, we obtain
\begin{equation} \label{compinv-proj}
\na_a \si + Z_a \si = -\ol{W}{}^{rs}{}_a{}^t F_{rst} \quad \text{for} \quad 
Z_a := -\ol{W}{}^{rs}{}_a{}^t Y_{trs} \in \cE_a.
\end{equation}
By direct computation, we have $\wh{Z}_a = Z_a - \Up_a$ under the projective transformation~\eqref{na-trans-proj}.
Further, a similar computation as in Section~\ref{s.invariant} shows that the projective equation $E_0(\si)=0$
has a nontrivial solution if and only if $(dZ)_{ab}=0$ and $\Ph(Z)_{ab}=0$ where
\begin{align*}
& (dZ)_{ab} = 2\na_{[a}Z_{b]} \in \cE_{[ab]},  \\
& \Ph(Z)_{ab} := \na_{(a}Z_{b)} - \Rho_{(ab)} - Z_{(a}Z_{b)}  \in \cE_{(ab)}.
\end{align*}

Next we go through conformally invariant operators in Section~\ref{s.invariant} and discuss their projective analogues. First, (exterior) covariant derivatives $\wt{\na}$ and $\newd$ are given exactly as in~\eqref{newna} and~\eqref{newd}. Projectively invariant operators 
$C_1$ and $D_1$ can be obtained by a slight modification of~\eqref{C1} and~\eqref{D1}, namely
\begin{align}
\label{C1-proj}
C_1&: \cE_{(ab)}[1] \to \cE_{a}[1], &
\tau_{ab} &\mapsto -2\ol{W}{}^{rs}{}_a{}^t \na_{[r} \tau_{s]t}, \\
\label{D1-proj}
D_1&: \cE_b[1] \to \cE_{(ab)}[1], &
\tau_b &\mapsto \na_{(a} \tau_{b)} - Z_{(a}  \tau_{b)}.
\end{align}
Now it is easy to verify relations summarized in~\eqref{comm} and~\eqref{obstructE}
hold also in the projective case. Moreover, the operator $\ol{D}$ defined by~\eqref{Dbar}
is projectively invariant and we obtain also a projective modification of~\eqref{CDcomp}:
\begin{align} \label{CDcomp-proj}
\bigl( (C_1 \circ D_1)\ta \bigr)_a &=  \tau_a - (H'_1(\newd \tau))_{a} \!-\! \ol{W}{}^{rs}{}_a{}^p
\bigl[ 2 \Ph(Z)_{pr} \tau_s + \tfrac12 (dZ)_{rs} \tau_p \bigr] ,
\\ \notag &\qquad \text{where} \\
\label{H'_1-proj}
(H'_1 \psi)_a &= \tfrac{1}{2} \ol{W}{}^{rs}{}_a{}^p (\ol{D} \psi)_{prs} ,
\end{align}
for $\ta_a \in \cE_a[1]$, $\psi_{ab} \in \cE_{[ab]}[1]$.
Projective invariance of $\ol{D}$ means that also $H'_1$ is projectively invariant.

\vspace{1ex}

\begin{theorem} \label{thm:onesol-proj-cc}
In dimension $n\ge 3$, for the projective-to-Ricci-flat operator
$E_0$~\eqref{E-proj} of a generic projective class $[\na]$,
its compatibility complex is given by the top line
of~\eqref{eq:onesol-equiv}, provided that
\begin{align*}
	V_0 &= \cE[1] , &
	V_1 &= \cE_{(ab)}[1], &
	V'_i &= \cE_{[a_1\cdots a_i]} [1] \\
	K_0 &= E_0, &
	K'_i &= \newd 
\end{align*}
and operators $C_1$ and $D_1$ are given by~\eqref{C1-proj}
and~\eqref{D1-proj}, with $H'_1$ given by~\eqref{H'_1-proj}.
\end{theorem}
\begin{proof}
Again, the conclusion follows from Lemma~\ref{lem:onesol-equiv}. We need
only check the commutativity of the first square
of~\eqref{eq:onesol-equiv}, which is ensured by the projective analog of
identities~\eqref{comm}, and the identity~\eqref{eq:onesol-H1'}
involving $H'_1$, which is now a rearrangement of
identity~\eqref{CDcomp-proj}.
\end{proof}

For dimension $n=2$, $(M,[\na])$ is either projectively flat or there is
only trivial solution of projective $E_0$. In the latter case, the
projective Cotton tensor $Y_{abc}$ is nonzero and one can easily
reproduce the discussion in Remark~\ref{lowdim}(1).

\appendix

\section{Lorentzian signature in 4-dimensions} \label{app:lorentzian}

\newcommand{\m}{\bar{m}}

We have noted in the text that the choice~\eqref{eq:Winv-from-V} of the
inversion $\ol{W}{}^{abcd}$, combined with the 4-dimensional
identity~\eqref{4dim} in Remark~\ref{lowdim}(2), works for any generic Weyl tensor of a Riemannian
metric. That is because the quadratic invariant $|W|^2 = W_{abcd}
W^{abcd} \ne 0$ if $W_{abcd} \ne 0$. The last implication no longer
holds for indefinite signature metrics. But does then the
inversion~\eqref{eq:Winv-from-V} necessarily fail? Remark~\ref{lowdim}(3) shows that more complicated canonical inversions (where $\ol{W}{}^{abcd}$ is covariantly and algebraically determined from $W_{abcd}$) may work when $|W|^2=0$. However, in this Appendix, we
find explicit examples in Lorentzian signature where all such canonical inversions fail.

When dealing with a 4-dimensional conformal structure $(M,[g])$, where
the signature of a metric representative $g$ is Lorentzian
$({-}{+}{+}{+})$, it is sometimes convenient to use an double null frame
$\{l,n,m,\m\}$ adapted to $g$, where $l,n$ are real null vectors, while
$m,\m$ are complex conjugate null vectors, with the only non-vanishing
inner products $g(l,n) = 1$ and $g(m,\m) = -1$. We are following the
conventions of~\cite[\S 8]{chandra-bh}. The structural properties of
this frame are preserved by the boosts in the $(l,n)$-plane and
rotations in the $m,\bar{m}$ plane. As such we can assign the following
(boost, spin) weights to the frame vectors: $|l| = (1,0)$, $|n| =
(-1,0)$, $|m| = (0,1)$, $|m| = (0,-1)$, which are then preserved by
tensor products, permutations and metric contractions. The boost and
spin weights naturally extend to the elements of the frame adapted basis
of any tensor space.

In 4-dimensions, the bundle $\cW_{pqrs}$ of 4-tensors with the same
structure as the Weyl tensor $W_{pqrs}$ (the same permutation symmetries
as the Riemann tensor and traceless) has rank 10 and the fibers can be
conveniently parametrized by the following 5 complex
\emph{Newman-Penrose scalars}~\cite[Eq.(1.294)]{chandra-bh}
\begin{equation}\label{eq:np-weyl}\begin{aligned}
	\Psi_0 &= -W_{pqrs} l^p      m^q       l^r       m^s , \\
	\Psi_1 &= -W_{pqrs} l^p      n^q       l^r       m^s , \\
	\Psi_2 &= -W_{pqrs} l^p      m^q  \m^r      n^s , \\
	\Psi_3 &= -W_{pqrs} l^p      n^q  \m^r      n^s , \\
	\Psi_4 &= -W_{pqrs} n^p \m^q      n^r  \m^s .
\end{aligned}\end{equation}
They are chosen so that when only one of these scalars is non-zero, the
corresponding Weyl tensor has pure boost weight, namely weight $-2$ for
$\Psi_0$, $-1$ for $\Psi_1$, $0$ for $\Psi_2$, $1$ for $\Psi_3$ and $2$
for $\Psi_4$. The \emph{Petrov classification}~\cite[\S 9]{chandra-bh}
of the algebraic type of the Weyl tensor corresponds to the following
filtration: Type~I ($\Psi_0=0$), Type~II ($\Psi_0=\Psi_1=0$), Type~III
($\Psi_0=\Psi_1=\Psi_2=0$) and Type~N ($\Psi_0= \Psi_1 = \Psi_2 = \Psi_3 =
0$).
The reconstruction of the original tensor from the
scalars~\eqref{eq:np-weyl} is given by~\cite[Eq.(1.298)]{chandra-bh}%
	\footnote{That formula has a typo; it is missing a factor of
	$\frac{1}{2}$ in front of its first term.}, %
but in a somewhat awkward form. We give an alternative formula below,
which is more convenient for our purposes.

We will use the following symmetrized and \emph{Kulkarni-Nomizu}
products
\begin{align}
\label{eq:sym-prod}
	(A B)_{pr} &= A_{(p} B_{r)} ,
	\\
\label{eq:kn-prod}
	(A \odot B)_{pqrs}
	&= A_{pr} B_{qs}
	 - A_{qr} B_{ps}
	 - A_{ps} B_{qr}
	 + A_{qs} B_{pr} .
\end{align}
If $A, B \in \cE_{(pr)}$, then $(A\odot B)_{pqrs}$ has the same index permutation symmetries as a Weyl tensor, with the exception of not automatically being traceless. 
A convenient reconstruction formula from the Newman-Penrose
scalars~\eqref{eq:np-weyl} is
\begin{scriptsize}
\begin{equation}
\begin{aligned}
	W
	&= -\Psi_0 (nn) \odot (\m\m) - \Psi_0^* (nn) \odot (mm)
	\\ &\quad {}
		+ 2\Psi_1   \big((nn) \odot (l\m) + (nm) \odot (\m\m)\big)
		+ 2\Psi_1^* \big((nn) \odot (lm) + (n\m) \odot (mm)\big)
	\\ &\quad {}
		- (\Psi_0 + \Psi_0^*)   \big((ll) \odot (nn) + (mm) \odot (\m\m) - 2 (ln) \odot (m\m)\big)
	\\ &\qquad {}
		+ 2i(\Psi_0 - \Psi_0^*) \big((lm) \odot (n\m) - (l\m) \odot (nm)\big)
	\\ &\quad {}
		+ 2\Psi_3^* \big((ll) \odot (n\m) + (lm) \odot (\m\m)\big)
		+ 2\Psi_3   \big((ll) \odot (nm) + (l\m) \odot (mm)\big)
	\\ &\quad {}
		-\Psi_4^* (ll) \odot (\m\m) - \Psi_4 (ll) \odot (mm)
	.
\end{aligned}
\end{equation}
\end{scriptsize}
The quadratic invariant of the Weyl tensor evaluates to
\begin{equation} \label{eq:np-weyl2}
	|W|^2
	= W_{pqrs} W^{pqrs}
	= 16 \Re (\Psi_0 \Psi_4 - 4 \Psi_1 \Psi_3 + 3 \Psi_2^2) .
\end{equation}
Since the $\Psi_i$ can be arbitrary complex numbers, we can definitely
find choices for which $W_{pqrs} \ne 0$, but $|W|^2=0$.
This means that the inversion
formula~\eqref{eq:Winv-from-V} for $\ol{W}{}^{pqrs}$, combined with
identity~\eqref{4dim}, cannot work. It remains to check whether there
exist any Weyl tensors that have $|W|^2=0$ and are generic in our sense
(Section~\ref{s.generic}).

Note that in the contraction~\eqref{eq:np-weyl2}, only those terms
survive that have total zero boost weight (since scalars can only have
vanishing boost weight). By that argument (or even more directly from formula~\eqref{eq:np-weyl2}), we can read off that tensors
of Type~III (which have only positive boost weight components) will be non-trivial but have $|W|^2 = 0$. The same argument implies that any scalar polynomial invariant built from a Type~III Weyl tensor will also vanish, including the cubic invariant $\overset{3}{V}{}_r^r$, meaning that the canonical inversion from Remark~\ref{lowdim}(3) also fails.

To verify genericity, We will next show
that the map $\mathrm{W} \colon \cE^s \to \cH_{pqr}[2] \subset \cE_{pqr}[2]$ for
Type~III tensors is injective when $\Psi_3\ne 0$, where $\cH_{pqr}[2]$
denotes the sub-bundle consisting of conformally weighted tensors
satisfying $H_{(pq)r} = H_{[pqr]} = 0$ and $g^{pr} H_{pqr} = 0$. This
will be obvious, once we find the matrix form of $\mathrm{W}$ in a basis
adapted to our double null frame.

It is convenient to introduce the special wedge product
\begin{equation} \label{eq:wedge-prod}
	(A \wedge B)_{pqr} = A_{pr} B_q - A_{qr} B_q .
\end{equation}
If $A \in \cE_{(pr)}$, then $(A \wedge B)_{(pq)r} = (A \wedge B)_{[pqr]}
= 0$. A convenient basis for $\cH_{pqr}[2]$ is the following, where we
group its elements by the indicated boost weight:
\begin{equation} \label{eq:hook-basis}
\begin{array}{rccccc}
	-2\colon & (nn)\wedge\m, & (nn)\wedge m, \\
	-1\colon & (\m\m)\wedge n, &
		2 (nm)\wedge\m + (nn)\wedge l, \\ &
		2 (n\m)\wedge m + (nn)\wedge l, &
		(mm)\wedge n, \\
	 0\colon &
		 2 (l\m)\wedge n + (\m\m)\wedge m, &
		 2 (n\m)\wedge l + (\m\m)\wedge m, \\ &
		 2 (lm)\wedge n + (mm)\wedge\m, &
		 2 (nm)\wedge l + (mm)\wedge\m \\
	 1\colon & (\m\m)\wedge l &
		2 (lm)\wedge\m + (ll)\wedge n, \\ &
		2 (l\m)\wedge m + (ll)\wedge n, &
		 (mm)\wedge l, \\
	 2\colon & (ll)\wedge\m & (ll)\wedge m ,
\end{array}
\end{equation}
where we will order the basis by boost weight with respect to the above
table, in row major order and from left to right within rows.

Since the map $\mathrm{W}$ is defined by a contraction, it preserves
boost weights. Moreover, if only the $\Psi_3$ and $\Psi_4$ scalars
parametrizing $W_{pqrs}$ are non-vanishing, contracting with the Weyl
tensor can only increase the boost weight by $1$ or $2$ respectively.
Clearly, the matrix form of $\mathrm{W}$ will be block triangular with
respect to subspaces of pure boost weight and checking injectivity will
come down to checking the injectivity of the blocks on the leading
diagonal. Parametrizing vectors as $v^s = v_l l^s + v_m m^s + v_m^* \m^s
+ v_n n^s$ and expressing the contraction $W_{pqrs} v^s$ in the
basis~\eqref{eq:hook-basis} shows that
\begin{equation} \label{eq:W-matrix}
	\mathrm{W} \begin{bNiceArray}{c}[margin]
		v_n \\ \midrule v_m^* \\ v_m \\ \midrule v_l
	\end{bNiceArray}
	= \begin{bNiceArray}{c|cc|c}[margin]
			& & & \\
			& & & \\ \midrule
			& & & \\
			& & & \\
			& & & \\
			& & & \\ \midrule
			0 & & & \\
			\Psi_3^* & & & \\
			0 & & & \\
			\Psi_3 & & & \\ \midrule
			-\Psi_4^* & -\Psi_3^* & 0 & \\
			0 & 0 & -\Psi_3^* & \\
			0 & -\Psi_3 & 0 & \\
			-\Psi_4 & 0 & -\Psi_3 & \\ \midrule
			& 0 & \Psi_4^* & \Psi_3^* \\
			& \Psi_4 & 0 & \Psi_3
		\end{bNiceArray}
		\begin{bNiceArray}{c}[margin]
			v_n \\ \midrule v_m^* \\ v_m \\ \midrule v_l
		\end{bNiceArray}
       .
\end{equation}
Indeed we see that the matrix in~\eqref{eq:W-matrix} is injective
exactly when $\Psi_3 \ne 0$. This means that a typical Type~III Weyl
tensor is generic in our sense, while having $|W|^2 = 0$. We can also
immediately see that Type~N Weyl tensors are not generic.

\bibliographystyle{utphys-alpha}
\bibliography{CC-conformal}

\providecommand{\href}[2]{#2}\begingroup\raggedright\begin{thebibliography}{10}

\bibitem{aabkhw}
S.~Aksteiner, L.~Andersson, T.~B{\"a}ckdahl, I.~Khavkine, and B.~Whiting,
  ``Compatibility complex for black hole spacetimes,''
  \href{http://dx.doi.org/10.1007/s00220-021-04078-y}{{\em Communications in
  Mathematical Physics} {\bfseries 384} (2021) 1585--1614},
  \href{http://arxiv.org/abs/1910.08756}{{\ttfamily arXiv:1910.08756}}.

\bibitem{ABB}
L.~Andersson, T.~B\"ackdahl, and P.~Blue, ``Second order symmetry operators,''
  {\em Classical and Quantum Gravity} {\bfseries 31} no.~13, (2014) 135015, 38.

\bibitem{BEG}
T.~N. Bailey, M.~G. Eastwood, and A.~R. Gover, ``Thomas's structure bundle for
  conformal, projective and related structures,''
  \href{http://dx.doi.org/10.1216/rmjm/1181072333}{{\em The Rocky Mountain
  Journal of Mathematics} {\bfseries 24} (1994) 1191--1217}.

\bibitem{Besse}
A.~L. Besse, \href{http://dx.doi.org/10.1007/978-3-540-74311-8}{{\em Einstein
  manifolds}}, vol.~10 of {\em Ergebnisse der Mathematik und ihrer Grenzgebiete
  (3)}.
\newblock Springer-Verlag, Berlin, 1987.

\bibitem{BC}
B.~D. Boe and D.~H. Collingwood, ``A comparison theory for the structure of
  induced representations. {II},''
  \href{http://dx.doi.org/10.1007/BF01159158}{{\em Mathematische Zeitschrift}
  {\bfseries 190} (1985) 1--11}.

\bibitem{Brinkmann}
H.~W. Brinkmann, ``Einstein spaces which are mapped conformally on each
  other,'' \href{http://dx.doi.org/10.1007/BF01208647}{{\em Mathematische
  Annalen} {\bfseries 94} (1925) 119--145}.

\bibitem{bcggg}
R.~L. Bryant, S.~S. Chern, R.~B. Gardner, H.~L. Goldschmidt, and P.~A.
  Griffiths, {\em Exterior Differential Systems}, vol.~18 of {\em Mathematical
  Sciences Research Institute Publications}.
\newblock Springer, New York, 1991.

\bibitem{calabi}
E.~Calabi, \href{http://dx.doi.org/10.1090/pspum/003}{``On compact,
  {R}iemannian manifolds with constant curvature. {I},''} in {\em Differential
  Geometry}, C.~B. Allendoerfer, ed., vol.~3 of {\em Proceedings of Symposia in
  Pure Mathematics}, pp.~155--180.
\newblock AMS, Providence, RI, 1961.

\bibitem{CD}
D.~M.~J. Calderbank and T.~Diemer, ``Differential invariants and curved
  {B}ernstein-{G}elfand-{G}elfand sequences,''
  \href{http://dx.doi.org/10.1515/crll.2001.059}{{\em Journal f\"ur die Reine
  und Angewandte Mathematik} {\bfseries 537} (2001) 67--103}.

\bibitem{chandra-bh}
S.~Chandrasekhar, {\em The mathematical theory of black holes}, vol.~69 of {\em
  International series of monographs on physics}.
\newblock Clarendon Press, 1983.

\bibitem{celm-riemann}
F.~Costanza, M.~Eastwood, T.~Leistner, and B.~McMillan, ``A {Calabi} operator
  for {Riemannian} locally symmetric spaces,'' 2021.
\newblock \href{http://arxiv.org/abs/2112.00841}{{\ttfamily arXiv:2112.00841}}.

\bibitem{celm-lorentz}
F.~Costanza, M.~Eastwood, T.~Leistner, and B.~McMillan, ``The range of a
  connection and a {Calabi} operator for {Lorentzian} locally symmetric
  spaces,'' 2023.
\newblock \href{http://arxiv.org/abs/2302.04480}{{\ttfamily arXiv:2302.04480}}.

\bibitem{eastwood-calabi}
M.~Eastwood, ``Variations on the de {R}ham complex,'' {\em Notices of the
  American Mathematical Society} {\bfseries 46} (1999) 1368--1376.
  \url{http://www.ams.org/notices/199911/fea-eastwood.pdf}.

\bibitem{edgar-hoglund}
S.~B. Edgar and A.~H\"oglund, ``Dimensionally dependent tensor identities by
  double antisymmetrization,'' \href{http://dx.doi.org/10.1063/1.1425428}{{\em
  Journal of Mathematical Physics} {\bfseries 43} (2002) 659--677},
  \href{http://arxiv.org/abs/gr-qc/0105066}{{\ttfamily arXiv:gr-qc/0105066}}.

\bibitem{cEspaces}
A.~Garcia-Parrado, J.~Herrera, and M.~Vadillo, ``Conformal {Einstein} spaces
  and conformally covariant operators,'' 2026.
\newblock \href{http://arxiv.org/abs/2601.17044}{{\ttfamily arXiv:2601.17044
  [math.DG]}}.
\newblock \url{https://arxiv.org/abs/2601.17044}.

\bibitem{Glap}
A.~R. Gover, ``Laplacian operators and {$Q$}-curvature on conformally
  {E}instein manifolds,''
  \href{http://dx.doi.org/10.1007/s00208-006-0004-z}{{\em Mathematische
  Annalen} {\bfseries 336} (2006) 311--334}.

\bibitem{GoNu}
A.~Gover and P.~Nurowski, ``Obstructions to conformally {Einstein} metrics in
  $n$ dimensions,''
  \href{http://dx.doi.org/10.1016/j.geomphys.2005.03.001}{{\em J.\ Geom.\
  Phys.} {\bfseries 56} (2006) 450--484},
  \href{http://arxiv.org/abs/math/0405304}{{\ttfamily arXiv:math/0405304}}.

\bibitem{kh-calabi}
I.~Khavkine, ``The {Calabi} complex and {Killing} sheaf cohomology,''
  \href{http://dx.doi.org/10.1016/j.geomphys.2016.06.009}{{\em Journal of
  Geometry and Physics} {\bfseries 113} (2017) 131--169},
  \href{http://arxiv.org/abs/1409.7212}{{\ttfamily arXiv:1409.7212}}.

\bibitem{kh-compat}
I.~Khavkine, ``Compatibility complexes of overdetermined {PDEs} of finite type,
  with applications to the {Killing} equation,''
  \href{http://dx.doi.org/10.1088/1361-6382/ab329a}{{\em Classical and Quantum
  Gravity} {\bfseries 36} (2019) 185012},
  \href{http://arxiv.org/abs/1805.03751}{{\ttfamily arXiv:1805.03751}}.

\bibitem{Kostant}
B.~Kostant, ``Lie algebra cohomology and the generalized {B}orel-{W}eil
  theorem,'' \href{http://dx.doi.org/10.2307/1970237}{{\em Annals of
  Mathematics. (2)} {\bfseries 74} (1961) 329--387}.

\bibitem{MRS}
J.-P. Michel, F.~Radoux, and J.~\v{S}ilhan, ``Second order symmetries of the
  conformal {L}aplacian,'' \href{http://dx.doi.org/10.3842/SIGMA.2014.016}{{\em
  SIGMA. Symmetry, Integrability and Geometry. Methods and Applications}
  {\bfseries 10} (2014) 016}.

\bibitem{seiler-inv}
W.~M. Seiler, {\em Involution: The Formal Theory of Differential Equations and
  its Applications in Computer Algebra}, vol.~24 of {\em Algorithms and
  Computation in Mathematics}.
\newblock Springer, 2010.

\bibitem{spencer}
D.~C. Spencer, ``Overdetermined systems of linear partial differential
  equations,'' \href{http://dx.doi.org/10.1090/s0002-9904-1969-12129-4}{{\em
  Bulletin of the American Mathematical Society} {\bfseries 75} (1969)
  179--240}.

\bibitem{tarkhanov}
N.~N. Tarkhanov, \href{http://dx.doi.org/10.1007/978-94-011-0327-5}{{\em
  Complexes of Differential Operators}}, vol.~340 of {\em Mathematics and Its
  Applications}.
\newblock Kluwer, Dordrecht, 1995.

\bibitem{CSS}
A.~\v{C}ap, J.~Slov\'ak, and V.~Sou\v{c}ek, ``Bernstein-{G}elfand-{G}elfand
  sequences,'' \href{http://dx.doi.org/10.2307/3062111}{{\em Annals of
  Mathematics. (2)} {\bfseries 154} (2001) 97--113}.

\bibitem{CS}
A.~\v{C}ap and V.~Sou\v{c}ek, ``Curved {C}asimir operators and the {BGG}
  machinery,'' \href{http://dx.doi.org/10.3842/SIGMA.2007.111}{{\em SIGMA.
  Symmetry, Integrability and Geometry. Methods and Applications} {\bfseries 3}
  (2007) 111}.

\bibitem{weibel}
C.~A. Weibel, \href{http://dx.doi.org/10.1017/CBO9781139644136}{{\em An
  introduction to homological algebra}}, vol.~38 of {\em Cambridge Studies in
  Advanced Mathematics}.
\newblock Cambridge University Press, Cambridge, 1994.

\end{thebibliography}\endgroup

\end{document}